\documentclass[oneside,english]{amsart}
\usepackage[T1]{fontenc}
\usepackage[latin9]{inputenc}
\usepackage{color}
\usepackage{amssymb}

\makeatletter
\numberwithin{equation}{section} 
\numberwithin{figure}{section} 
  \theoremstyle{plain}
  \newtheorem*{thm*}{Theorem}
  \theoremstyle{plain}
  \newtheorem{thm}{Theorem}[section]
  \theoremstyle{definition}
  \newtheorem{defn}[thm]{Definition}
  \theoremstyle{plain}
  \newtheorem{lem}[thm]{Lemma}
  \theoremstyle{plain}
  \newtheorem{prop}[thm]{Proposition}
  \theoremstyle{remark}
  
  \theoremstyle{remark}
  \newtheorem*{acknowledgement*}{Acknowledgement}

\@addtoreset{equation}{section}

\usepackage{babel}
\makeatother

\begin{document}

\title[A setting for higher order differential equation fields and
higher order Lagrange and Finsler spaces] {A setting for higher
  order differential equation fields and
higher order Lagrange and Finsler spaces}

\author[Bucataru]{Ioan Bucataru}
\address{Ioan Bucataru, Faculty of Mathematics, University
  Alexandru Ioan Cuza, Ia\c si, 700506, Romania}
\urladdr{http://www.math.uaic.ro/\textasciitilde{}bucataru/}

\date{\today}

\begin{abstract}
We use the Fr\"olicher-Nijenhuis formalism to reformulate the inverse
problem of the calculus of variations for a system of differential
equations of order $2k$ in terms of a semi-basic $1$-form of order
$k$.  Within this general context, we use the homogeneity proposed by
Crampin and Saunders in \cite{CS11} to formulate and discuss the
projective metrizability problem for higher order differential
equation fields. We provide necessary and sufficient conditions for
higher order projective metrizability in terms of homogeneous
semi-basic $1$-forms. Such a semi-basic $1$-form is the
Poincar\'e-Cartan $1$-form of a higher order Finsler function, while
the potential of such semi-basic $1$-form is a higher order Finsler function.  
\end{abstract}



\maketitle

\section{Introduction}

The Fr\"olicher-Nijenhuis formalism is a very useful tool for developing a differential
calculus that provides a geometric setting for studying differential equations
fields, \cite{BCD11, BD09, CSC86, GM00, deleon85, saunders02, szilasi03}. 

The framework for studying higher order differential equation fields, on a
configuration manifold $M$, is the higher order tangent bundle $T^rM$, for some
natural number $r\geq 1$. In Section \ref{sec:vdc} we discuss some geometric
structures that naturally live on higher order tangent bundles:
vertical distributions, Liouville vector fields, tangent
structures. We use the Fr\"olicher-Nijenhuis formalism associated to these
geometric structures to provide a vertical differential calculus,
which is very useful for studying higher order differential equation fields. Motivated
by the foliated structure of the higher order tangent bundles, we show
that vertical vector fields, as well as their dual, semi-basic $1$-forms,
play an important role in the vertical differential calculus, which we
associate to a higher order differential equations field. We will use the formalism developed in Subsection
\ref{subsec:hodef} and especially Lemma \ref{lem:cas} in
Sections \ref{sec:ls} and \ref{sec:pm} to characterize those
differential equation fields that may be associated to a variational
problem of a Lagrange or a Finsler space of higher order. 

The inverse problem of the calculus of variations requires to
determine the necessary and sufficient conditions such that a system
of ordinary differential equations, of order $2k$, may be derived from a
variational problem. For $k=1$, these conditions can be formulated in terms of
a multiplier matrix \cite{krupkova97, morandi90, PR11, sarlet82,
  saunders89}, a closed $2$-form \cite{AT92, crampin81}, or a
semi-basic $1$-form \cite{BD09}. The approach, based on the existence
of a closed $2$-form, developed by Crampin in \cite{crampin81}, was
extended by de Le\'on and Rodrigues in \cite{deLeon92} for $k>1$. A
deep relationship between variational equations of arbitrary order
and closed $2$-forms has been found and studied by Krupkov\'a in
\cite{krupkova86, krupkova87}. In
Section \ref{sec:ls} we use the vertical differential calculus, which we develop in Section \ref{sec:vdc}, to
provide global formulations for the geometric structures one can
associate to higher order Lagrangians and higher order differential
equation fields. In Theorem \ref{thm:ls} we reformulate
the inverse problem of the calculus of variations in terms of a
semi-basic $1$-form of order $k$. For the variational case, we show
that such a semi-basic $1$-form is the Poincar\'e-Cartan $1$-form of a
Lagrangian of order $k$. In Proposition \ref{prop:homls} we prove that
some homogeneity properties of a regular Lagrangian transfer to its
canonical Euler-Lagrange vector field. 

An important case of the inverse problem of the calculus
of variations refers to homogeneous systems of ordinary differential equations. For $k=1$,
this problem contains what is known as the projective metrizability problem or ''the
Finslerian version of Hilbert's fourth problem'',
\cite{alvarez05, crampin08, cms12, saunders12, szilasi07}. The projective metrizability problem requires to
determine if the solutions of a homogeneous system of
second order ordinary differential equations coincide with the
geodesics of a Finsler metric, up to an orientation preserving
reparameterization \cite{BM11a, BM11b, shen01, szilasi02}. For the case $k>1$, an attempt to address and
study the projective metrizability problem, requires first a good definition
of homogeneity for systems of higher order differential equations as
well as for higher order Lagrangians. In this work we use the
definitions of homogeneity proposed by Crampin and Saunders in
\cite{CS11} to formulate and study the projective metrizability
problem in Section \ref{sec:pm}. In Subsection \ref{subsec:hofs} we
introduce and discuss higher order Finsler spaces. We show that the
regularity condition, which we propose for a higher order Finsler function, is
equivalent to the regularity condition proposed by Crampin and
Saunders in \cite{CS11} for parametric Lagrangians and that it reduces, when
$k=1$, to the classic regularity condition of a Finsler function.  
We show that the variational problem of a higher order Finsler
function uniquely determines a projective class of homogeneous
differential equation fields. In Theorem \ref{thm:pm} we
characterize the projective metrizability problem of a homogeneous
differential equation field of order $2k$ in terms of a
homogeneous semi-basic $1$-form of order $k$. We prove that, similarly
with what happens in
the case $k=1$, such a semi-basic $1$-form is the Poincar\'e-Cartan
$1$-form of a Finsler function of order $k$. Moreover, the potential of such homogeneous semi-basic $1$-form
is a Finsler function of order $k$ that metricizes the equation
field. 

In the last section we discuss some examples of higher order
differential equation fields and their relations with higher order
Lagrange and Finsler spaces. It has been shown in \cite{CMOP06} that
biharmonic curves, which are solutions of a fourth order differential
equations field, are solutions of the Euler-Lagrange equations for a
regular Lagrangian $L_2$ of order $2$. See also \cite{BCD11} for a
different approach. We use the homogeneity properties of the second
order Lagrangian $L_2$ to obtain some information for the
corresponding Euler-Lagrange vector field (biharmonic differential
equations field).  We provide an example of a second order Finsler functions, which in the Euclidian context reduces to the
parametric Lagrangian $L$, studied by Crampin and Saunders in
\cite[\S 6]{CS11}.

\section{Vertical differential calculus on higher order tangent
  bundles} \label{sec:vdc}

In this section we discuss first some geometric structures that
are naturally defined on higher order tangent spaces: vertical
distributions, Liouville vector fields, tangent structures, higher
order differential equation fields. We use these
geometric structures and the corresponding differential calculus induced
by the Fr\"olicher-Nijenhuis formalism to develop a
geometric setting, which we will use in Sections \ref{sec:ls} and
\ref{sec:pm} to discuss two important problems associated to a (homogeneous) higher order
differential equation field. 

\subsection{Geometric structures on higher order tangent bundles}

In this work $M$ is a real, $n$-dimensional and $C^{\infty}$-smooth manifold. We
will assume that all objects are smooth where
defined. We denote the ring of smooth functions on $M$ by
$C^{\infty}(M)$, while the Lie algebra of vector fields on $M$ is denoted by
$\mathfrak{X}(M)$.  

The framework to develop a geometric setting for studying systems of higher order ordinary
differential equations on a manifold $M$ is the higher order tangent bundle
$T^{r}M=J_{0}^{r}M$, for some $r \in\mathbb{N}^*$, \cite{AT92, CSC86, deleon85, miron94,
miron97, tulczyjew76}. This is the jet bundle of order $r$, of
curves $c$ from a neighborhood of $0$ in $\mathbb{R}$ to $M$. For a
curve $c:I \to M$, $c(t)=(x^i(t))$, consider $j^rc: I \to T^rM$
its jet lift of order $r$. If $(x^{i})$ are local coordinates on $M$, the induced
local coordinates on $T^rM$ are denoted by
$(x^{i},y^{(1)i},\cdots,y^{(r)i})$, where
\[
y^{(\alpha)i}(j_{0}^{r}c)=\left.\frac{1}{\alpha!}\frac{d^{\alpha}(x^{i}(c(t))}{dt^{\alpha}}\right\vert
_{t=0},\quad \alpha\in\{1,..,r\} .\] Let $y^{(0)i}:=x^i$ and 
denote $T^0M=M$. The canonical submersion $\pi^r_{\alpha}: T^rM \to
T^{\alpha}M$, for each $\alpha \in \{0,1,...,r-1\}$,  induces a
natural foliation of $T^rM$. We will consider also the subbundle
$T_0^rM=\{(x,y^{(1)},\cdots,y^{(r)})\in T^rM,
y^{(1)}\neq 0\}$. It follows that $T_0^rM =
\left(\pi^r_1\right)^{-1}\left(T_0M\right)$.

A curve $c:I \to M$ is called a \emph{regular
curve} if $j^rc(t) \in T^r_0M$ for all $t\in I$ and some $r \in\mathbb{N}^*$.

The \emph{tangent structure} (or vertical endomorphism) of order
$r$ is the $(1,1)$-type tensor field on $T^{r}M$ defined as
\begin{equation} J=\frac{\partial}{\partial y^{(1)i}}\otimes
dx^{i}+\frac{\partial}{\partial y^{(2)i}}\otimes
dy^{(1)i}+\cdots+\frac{\partial}{\partial y^{(r)i}}\otimes
dy^{(r-1)i}.\label{eq:J}\end{equation} For each $\alpha\geq 2$, we will consider $J^{\alpha}$, the composition of $J$,
$\alpha$-times. The following
properties are straightforward: $J^{r+1}=0,\,\,
\operatorname{Im}J^{\alpha}=\operatorname{Ker}J^{r-\alpha+1},
\alpha\in\{1,...,r\}$.

The foliated structure of $T^{r}M$ gives rise to $r$ regular vertical
distributions
\begin{eqnarray*}
V_{\alpha}(u)  =\operatorname{Ker} D_{u}\pi_{\alpha-1}^{r}
=\operatorname{Im} J_u^{\alpha} =\operatorname{Ker}
J_u^{r-\alpha+1}, \textrm{ for } u\in T^rM,  \alpha
\in\{1,...,r\}.\end{eqnarray*} 
Each distribution $V_{\alpha}$, for $\alpha\in\{1,...,r\},$ is tangent to the fibers of
$\pi_{\alpha-1}^{r}: (x^{i},y^{(1)i},\cdots,y^{(r)i})
\rightarrow(x^{i},y^{(1)i},\cdots,y^{(\alpha-1)i})$, and hence it
is integrable. We have that $\dim V_{\alpha}=(r-\alpha+1)n$,
$\alpha\in\{1,...,r\}$ and $V_{r}(u)\subset
V_{r-1}(u)\subset\cdots\subset V_{1}(u),$ for each $u\in T^rM$. We
will denote by $\mathfrak{X}^{V_{\alpha}}(T^rM)$ the Lie subalgebra of
vertically valued vector fields. 

An important set of vertical vector fields is provided by the
\emph{Liouville vector fields} (or dilation vector fields) $\mathbb{C}_{\alpha}\in \mathfrak{X}^{V_{\alpha}}(T^rM)$,
$\alpha \in \{1,..., r\}$. These vector fields are locally given by:
\begin{eqnarray} \mathbb{C}_{\alpha} =
  y^{(1)i}\frac{\partial}{\partial y^{(\alpha)i}} + 2y^{(2)i}
    \frac{\partial}{\partial y^{(\alpha+1)i}} + \cdots +
      (r+1-\alpha)y^{(r+1-\alpha)i}\frac{\partial}{\partial
        y^{(r)i}}, \ \alpha \in \{1,...,
      r\}. \label{calpha} \end{eqnarray}
For the Liouville vector fields, we have the following formulae for
their Lie brackets
\begin{equation}
\begin{array}{lcl}
 [ \mathbb{C}_{\alpha}, \mathbb{C}_{\beta}] & =  & \left\{ 
\begin{array}{ll} (\alpha -\beta) \mathbb{C}_{\alpha+\beta-1}, &
  \textrm{if } \alpha +\beta \leq r-1, \\
0,  & \textrm{otherwise}. \end{array}  \right.
\end{array} \label{cacb}
\end{equation}

We will make use of the Fr\"olicher-Nijenhuis formalism, 
\cite{frolicher56, GM00, KMS93}, to develop a
differential calculus that will be useful to address various problems
associated to a differential equation field, \cite{BCD11, BD09, klein68, szilasi03}. For a vector valued
$l$-form $L$  on $T^rM $ consider the derivation of degree $l-1$, $i_L:\Lambda^q(T^rM) \to
\Lambda^{q+l-1}(T^rM)$ and the derivation of degree $l$, $d_L:\Lambda^q(T^rM) \to
\Lambda^{q+l}(T^rM)$. These two
derivations are related by the
following formula
\begin{eqnarray}
d_L=i_L\circ d + (-1)^l d\circ i_L. \label{dil} \end{eqnarray}
For two vector valued forms $K$ and $L$ on $T^rM$, of
degree $k$ and  respectively
$l$, consider the \emph{Fr\"olicher-Nijenhuis bracket} $[K, L]$,
which is the vector valued $(k+l)$-form on $T^rM$,
uniquely defined by
\begin{eqnarray}
d_{[K,L]}= d_K\circ d_L -(-1)^{kl} d_L\circ d_K. \label{dkl} \end{eqnarray} 
The Fr\"olicher-Nijenhuis brackets of the Liouville vector fields
$\mathbb{C}_{\alpha}$ (vector valued $0$-forms) and the vertical
endomorphisms $J^{\beta}$  (vector valued $1$-forms) are the vector
valued $1$-forms given by the
following formulae
\begin{equation}
\begin{array}{lcl}
 [ \mathbb{C}_{\alpha}, J^{\beta}] & =  & \left\{ 
\begin{array}{ll}  -\beta J^{\alpha+\beta-1}, &
  \textrm{if } \alpha +\beta \leq r+1, \\
0,  & \textrm{otherwise}. \end{array}  \right.
\end{array} \label{cajb}
\end{equation}

For $r=1$, semi-basic $1$-forms have shown their usefulness to adress various
problems associated to second order differential equation fields,
\cite{BD09, BM11a, GM00}. We will see also that for $r>1$, semi-basic
$1$-forms, of some order, are useful to formulate a geometric setting
for higher order differential equation fields. These forms where
introduced and discussed in \cite[Def 1]{ALR91}. However, in our work
a semi-basic $1$-form of order $\alpha$ on $T^rM$ corresponds to what
is called in \cite{ALR91} a semi-basic $1$-form of order $r+1-\alpha$.
\begin{defn} \label{defn:semibasic} A form on $T^rM$ is called \emph{semi-basic of order} $\alpha\in
  \{1,...,r\}$ if it is semi-basic with respect to the
  submersion $\pi^r_{\alpha-1}$.
\end{defn}
A form $\theta$ on $T^rM$ is semi-basic of
order $\alpha$ if it vanishes whenever one of its argument is a
vertical vector field in $\mathfrak{X}^{V_{\alpha}}(T^rM)$. Therefore, $\theta\in
\Lambda^1(T^rM)$  is semi-basic of order $\alpha$ if and
only if $i_{J^{\alpha}}\theta=0$. 
Semi-basic $1$-forms of order $\alpha$ are the dual equivalent of vertical
vector fields in $\mathfrak{X}^{V_{\alpha}}(T^rM)$. Hence we have that $\theta\in
\Lambda^1(T^rM)$  is semi-basic of order $\alpha$ if and
only if  there exists $\eta\in \Lambda^1(T^rM)$ such that
$\theta=i_{J^{r-\alpha+1}}\eta = \eta \circ J^{r-\alpha+1}$. 
Locally, a $1$-form $\theta$ on $T^rM$  is semi-basic of order $\alpha$ if and
only if 
\begin{eqnarray}
\theta = \theta_{(0)i}dx^i + \theta_{(1)i}dy^{(1)i} + \cdots
+\theta_{(\alpha-1)i} dy^{(\alpha-1)i}, \label{sba1} \end{eqnarray}
where the $\alpha$ components $\theta_{(0)i}, ..., \theta_{(\alpha-1)i}$ are smooth functions defined on
domains of local charts on $T^rM $.

For a function $f\in C^{\infty}(T^rM)$ and $\alpha \in \{1,...,r\}$ we have that $d_{J^{\alpha}}f$ is a
semi-basic $1$-form of order $r-\alpha+1$.
For a function $f\in C^{\infty}(T^rM)$ we have that $df$ is a
semi-basic $1$-form of order $\alpha \in \{1,...,r\}$ if and only if
$f$ is constant along the fibers of the submersion $\pi^r_{\alpha-1}$
and hence one can restrict it to $T^{\alpha-1}M$.

\subsection{Higher order differential equation fields} \label{subsec:hodef}

A system of higher order differential equations, whose coefficients do
not depend explicitly on time, can be viewed as a special vector field
on some higher order tangent bundle.  For such systems, we will use the
definition for homogeneous differential equation fields of order $r$,
which was proposed by Crampin and Saunders in \cite{CS11}.   

As it happens in the case $r=1$, Liouville vector fields
$\mathbb{C}_{\alpha}$, are
important for defining the notion of homogeneity for various geometric
structures on $T^rM$. Whenever we want to consider homogeneous
structures, which are not necessarily polynomial in the fibre
coordinates, we will consider them 
defined on the subbundle $T^r_0M$. 
\begin{defn} \label{defn:rspray} Consider a vector field $S$ on $T^rM$. We say
    that $S$  is a \emph{semispray of order $r$} if it satisfies the
  condition $JS=\mathbb{C}_1$. \end{defn}

In induced coordinates for
$T^rM$, a semispray of order $r$ is given by
\begin{eqnarray}
S=y^{(1)i}\frac{\partial}{\partial
x^{i}}+2y^{(2)i}\frac{\partial}{\partial
y^{(1)i}}+\cdots+ry^{(r)i}\frac{\partial}{\partial
y^{(r-1)i}}-(r+1)G^{i}\frac{\partial}{\partial
y^{(r)i}},\label{rsemispray}\end{eqnarray} for some functions
$G^i$ defined on domains of induced local charts. 

Alternatively, we have that a vector field $S$ on $T^rM$  is a semispray of order $r$
if and only if  any integral curve of
$S$, $\gamma:I \to T_0^rM$, is of the form
$\gamma=j^r(\pi_0^r\circ \gamma)$. For an integral
curve $\gamma: I \to T^r_0M$ of $S$, we say that curve
$c=\pi_0^r\circ \gamma$ is a \emph{geodesic} of $S$. Therefore, a
regular curve $c:I \to M$ is a geodesic of $S$ if and only if
$S\circ j^rc=(j^rc)'$. Locally, a regular curve $c:I \to M$,
$c(t)=(x^i(t))$, is a geodesic of $S$ if and only if it satisfies
the system of $(r+1)$ order ordinary differential equations
\begin{eqnarray}
\frac{1}{(r+1)!}\frac{d^{r+1}x^i}{dt^{r+1}} +
G^{i}\left(x,\frac{dx}{dt},\cdots,\frac{1}{r!}\frac{d^{r}x}{dt^{r}}\right)=0.\label{rode}
\end{eqnarray}
Therefore semisprays of order $r$ describe systems of higher order
differential equations which have regular curves on $M$ as solutions.

We will consider also, $d_T$, the \emph{Tulczyjew differential
  operator} on $T^rM$, also called the total derivative operator, which is given by \cite{tulczyjew76}
\begin{eqnarray}
d_T=y^{(1)i}\frac{\partial}{\partial
x^{i}}+2y^{(2)i}\frac{\partial}{\partial
y^{(1)i}}+\cdots+ry^{(r)i}\frac{\partial}{\partial
y^{(r-1)i}}. \label{dtr}\end{eqnarray} Using the Tulczyjew operator, a
semispray $S$ of order $r$ can be written as follows
\begin{eqnarray} S=d_T -(r+1)G^{i}\frac{\partial}{\partial
y^{(r)i}}.\label{dtrs}\end{eqnarray} 
Differential operator $d_T$ maps a function $f\in
C^{\infty}(T^{\alpha}M)$, $\alpha\in \{0,..., r-1\}$, into a function
$d_Tf:=d_T(f\circ \pi^k_{\alpha})\in
C^{\infty}(T^{\alpha+1}M)$. The function $d_Tf$ is basic
with respect to the submersion $\pi^k_{\alpha+1}$, therefore we can assume that
it is defined on $T^{\alpha+1}M$ and hence $d_Tf \in
C^{\infty}(T^{\alpha+1}M)$. In view of formula \eqref{dtrs}, for an arbitrary semispray of order $r$, $S\in
\mathfrak{X}(T^rM)$, and a function $f\in
C^{\infty}(T^{\alpha}M)$, $\alpha\in \{0,..., r-1\}$, we have that
$Sf=d_Tf \in C^{\infty}(T^{\alpha+1}M)$. 

The Fr\"olicher-Nijenhuis brackets of an arbitrary semispray $S$ and the vertical
endomorphisms $J^{\alpha}$ are useful to fix a (multi) connection on
$T^rM$ \cite{ALR91, BCD11, CSC86, saunders02}. In this work we will use only the vertical valued components of these vector
valued $1$-forms. 
\begin{lem} \label{lem:cas} Consider $S$ a semispray of order $r$ and
  $\alpha \in \{1,..., r\}$.
\begin{itemize} 
\item[i)] The Lie brackets $[\mathbb{C}_{\alpha}, S]$ are given by
\begin{eqnarray} [\mathbb{C}_{\alpha}, S]=\alpha \mathbb{C}_{\alpha-1}
  + U_{\alpha}, \label{cas}\end{eqnarray}
where $\mathbb{C}_0=S$ and $U_{\alpha}\in \mathfrak{X}^{V_r}(T^rM)$.
\item[ii)] The vertical components of the
  Fr\"olicher-Nijenhuis brackets $[S, J^{\beta}]$ are given by 
\begin{equation}
\begin{array}{lcl}
 J^{\alpha} [S, J^{\beta}] & =  & \left\{ 
\begin{array}{ll}  -\beta J^{\alpha+\beta-1}, &
  \textrm{if } \alpha +\beta \leq r+1, \\
0,  & \textrm{otherwise}. \end{array}  \right.
\end{array} \label{jasjb}
\end{equation}
\item[iii)] For a semi-basic $1$-form $\theta\in \Lambda^1(T^rM)$, of
  order $\alpha$, we have 
\begin{eqnarray}
i_{[S, J^{\beta}]}\theta =-\beta i_{J^{\beta-1}}\theta, \quad \forall
\beta\in \{1,...,r\}. \label{isjbt} \end{eqnarray}
\item[iv)] Consider $\theta\in \Lambda^1(T^rM)$ a semi-basic $1$-form
  of order $\alpha$. Then
  $\mathcal{L}_{\mathbb{C}_{\beta}}\theta$ is also a semi-basic $1$-form of
  order $\alpha$, for all $\beta\in \{1,...,r\}$. 
\item[v)] Consider $\theta\in \Lambda^1(T^rM)$ a semi-basic $1$-form
  of order $\alpha$ such that
  $\mathcal{L}_{S}\theta-df$ is a semi-basic $1$-form of order $1$,
  for some function $f$ on $T^rM$. Then
  the function $f$ can be restricted to $T^{\alpha}M$ and the $1$-form
  $\theta$ satisfies the following formulae 
\begin{eqnarray}
i_{J^{\gamma}}\theta=\gamma ! \sum_{\beta=1}^{\alpha-\gamma}
\frac{(-1)^{\beta-1}}{(\beta + \gamma) !}\mathcal{L}_S^{\beta-1}d_{J^{\beta+\gamma}}f, \quad \forall \gamma
\in \{0,1,..., \alpha-1\}. \label{tabg}
\end{eqnarray}
\end{itemize}
\end{lem}
\begin{proof}
For $\beta \in \{1,...,r\}$ and for every $X \in
\mathfrak{X}(T^rM)$ we have
\begin{eqnarray} [S,J^{\beta}X] - J^{\beta}[S,X] + \beta
J^{\beta-1}X= -U_{\beta}  \in \operatorname{Ker}J = 
\operatorname{Im}J^r\label{sjbetax1},
\end{eqnarray}
which has been shown in \cite[(3.27)]{BCD11}. 

i) In formula \eqref{sjbetax1} we take $X=S$ and use
$J^{\beta}S=\mathbb{C}_{\beta}$, for all $\beta \in \{1,...,r\}$. It follows formula \eqref{cas}.

ii)  By composing
in formula \eqref{sjbetax1} to the left with $J^{\alpha}$, for $\alpha
\geq 1$, we obtain
formula \eqref{jasjb}.

iii) Consider $\theta \in \Lambda^1(T^rM)$ a semi-basic $1$-form of order
$\alpha$. It follows that there exists $\eta \in \Lambda^1(T^rM)$ such
that $\theta=i_{J^{r-\alpha+1}}\eta$. Using now formula \eqref{jasjb} we
obtain
\begin{eqnarray*}
i_{[S, J^{\beta}]}\theta = i_{J^{r-\alpha+1} [S, J^{\beta}]}\eta =-\beta
i_{J^{r-\alpha+\beta}} \eta = -\beta i_{J^{\beta-1}} \theta, 
\end{eqnarray*} which shows that formula \eqref{isjbt} is true.

iv) Since $\theta\in \Lambda^1(T^rM)$ is a semi-basic $1$-form
  of order $\alpha \in\{1,...,r\}$, it follows that
  $i_{J^{\alpha}}\theta=0$. Using the corresponding commutation rules
  and formulae \eqref{cajb} it follows
\begin{eqnarray*} 
i_{J^{\alpha}}\mathcal{L}_{\mathbb{C}_{\beta}}\theta= \mathcal{L}_{\mathbb{C}_{\beta}}
i_{J^{\alpha}} \theta + i_{[J^{\alpha}, \mathbb{C}_{\beta}]}\theta =
\alpha i_{J^{\alpha+\beta-1}}\theta = 0,
\end{eqnarray*}
which proves that $\mathcal{L}_{\mathbb{C}_{\beta}}\theta$ is a semi-basic
$1$-form of order $\alpha$.

v) We know that $\theta\in \Lambda^1(T^rM)$ is a semi-basic $1$-form
of order $\alpha$, which means $i_{J^{\alpha}}\theta=0$. Moreover,
there exists a function $f\in C^{\infty}(T^rM)$ such that
$\mathcal{L}_S\theta - df$ is a semi-basic $1$-form of order $1$,
which means that we have
\begin{eqnarray} i_J\mathcal{L}_S\theta =
  d_Jf. \label{ijlstdf} \end{eqnarray}
If we apply $i_{J^{\alpha-1}}$ to both sides of formula \eqref{ijlstdf} and use the commutation
rule we obtain
\begin{eqnarray} \mathcal{L}_Si_{J^{\alpha}}\theta +
  i_{[J^{\alpha},S]}\theta=df\circ J^{\alpha}. \label{lstdl2} \end{eqnarray}
Using formula \eqref{isjbt} it follows that $i_{[J^{\alpha}, S]}
\theta = \alpha i_{J^{\alpha-1}}\theta$. From 
formula \eqref{lstdl2} we have that that $\alpha
i_{J^{\alpha-1}}\theta  = i_{J^{\alpha}}df$. Consequently, we have  that $df\circ J^{\alpha+1}=0$, which means 
that $df\in \Lambda^{1}(T^rM)$ is a semi-basic $1$-form of order
$\alpha+1$. It follows that $f$ is constant on the fibres of $\pi^{r}_{\alpha}: T^{r}M \to
T^{\alpha}M$ and therefore, we can restrict $f$ to $T^{\alpha}M$ and assume that it is a
function defined on $T^{\alpha}M$.

We will prove now that $\theta$ satisfies formulae \eqref{tabg}. We have seen above that 
\begin{eqnarray}
i_{J^{\alpha-1}}\theta =\frac{1}{\alpha} d_{J^{\alpha}}f, \label{ijkt1}
\end{eqnarray}
which is formula \eqref{tabg} for $\gamma=\alpha-1$.
We apply $\mathcal{L}_S$ to both sides of this formula, use the
commutation rule, and obtain
\begin{eqnarray*}
i_{J^{\alpha-1}}\mathcal{L}_S\theta + i_{[S,J^{\alpha-1}]}\theta
=\frac{1}{\alpha}\mathcal{L}_Sd_{J^{\alpha}}f. \end{eqnarray*}
We use now formulae \eqref{ijlstdf}  and \eqref{isjbt} to
obtain 
\begin{eqnarray*}
d_{J^{\alpha-1}}f -(\alpha-1) i_{J^{\alpha-2}}\theta
=\frac{1}{\alpha}\mathcal{L}_Sd_{J^{\alpha}}f. \end{eqnarray*}
Above formula implies 
\begin{eqnarray}
i_{J^{\alpha-2}}\theta =\frac{1}{\alpha-1}d_{J^{\alpha-1}}f -
\frac{1}{\alpha(\alpha-1)}\mathcal{L}_S d_{J^{\alpha}}f, \label{ijkt2}
\end{eqnarray}
which is formula \eqref{tabg} for $\gamma=\alpha-2$.
We apply again $\mathcal{L}_S$ to both sides of the
above formula, use the commutation rule, and obtain 
\begin{eqnarray}
i_{J^{\alpha-3}}\theta =\frac{1}{\alpha-2}d_{J^{\alpha-2}}f -
\frac{1}{(\alpha-1)(\alpha-2)} \mathcal{L}_S d_{J^{\alpha -1}}f +
\frac{1}{\alpha(\alpha-1)(\alpha-2)}\mathcal{L}_S^2 d_{J^{\alpha}}f,\label{ijkt3}
\end{eqnarray}
which is formula \eqref{tabg} for $\gamma=\alpha-3$.
We continue the process and obtain
\begin{eqnarray}
i_{J}\theta =\frac{1}{2}d_{J^2}f - \frac{1}{2\cdot 3} \mathcal{L}_S d_{J^3}f + \cdots +
\frac{(-1)^{\alpha-2}}{2\cdots \alpha}\mathcal{L}_S^{\alpha-2} d_{J^{\alpha}}f. \label{ijkt4}
\end{eqnarray}
Formula \eqref{ijkt4} represents formula \eqref{tabg} for $\gamma=1$.
Now, for the last step we use above formula, formula
\eqref{isjbt}, for $\beta=1$, as well as formula \eqref{ijlstdf}:
\begin{eqnarray*}
\theta=-i_{[S,J]}\theta=i_J\mathcal{L}_S\theta -\mathcal{L}_Si_J\theta
= d_Jf- \mathcal{L}_Si_J\theta.
\end{eqnarray*}
It follows that $\theta$ is given by
formula 
\begin{eqnarray}
\theta=\sum_{\beta=1}^{\alpha} \frac{(-1)^{\beta-1}}{\beta
  !}\mathcal{L}_S^{\beta-1}d_{J^{\beta}}f, \label{tab}
\end{eqnarray}
which represents formula \eqref{tabg} for $\gamma=0$.
\end{proof}
Formula \eqref{jasjb} was proven for the case $\alpha=\beta=1$ in
\cite[Lemma 5.1]{CSC86}.

\begin{defn} \label{hom_rspray}
A semispray $S\in
  \mathfrak{X}(T_0^rM)$ of order $r$
 is called \emph{homogeneous} if the distribution
    $\mathcal{D}=\textrm{span} \{S, \mathbb{C}_1, ...., \mathbb{C}_r\}$
    is involutive. \end{defn}
Above definition of homogeneity has been proposed in \cite[Definition
3.1]{CS11}.  In view
of formulae \eqref{cas}, a semispray $S\in \mathfrak{X}(T^r_0M)$ of
order $r$, is homogeneous if and only if for the vertical vector fields
$U_{\alpha}\in \mathfrak{X}^{V_r}(T^r_0M)$, 
there exist the functions $P_{\alpha}\in C^{\infty}(T^r_0M)$ such that
$U_{\alpha}=P_{\alpha}\mathbb{C}_r$, for all $\alpha \in \{1,...,r\}$.
Therefore, a semispray $S$ of order $r$ is homogeneous if and only if there exists functions $P_{\alpha}\in
    C^{\infty}(T^r_0M)$, $\alpha \in \{1,...,r\}$, such
    that 
\begin{eqnarray} [\mathbb{C}_1, S]= S +P_1\mathbb{C}_r, \quad  [\mathbb{C}_{\alpha}, S]=\alpha
  \mathbb{C}_{\alpha-1}+ P_{\alpha}\mathbb{C}_r, \ \alpha \in
  \{2,...,r\}. \label{rhspray}\end{eqnarray}
If we write the Jacobi identities for the vector fields $S$,
$\mathbb{C}_1$, ...., $\mathbb{C}_r$, and use the above formulae, we obtain that functions
$P_1,..., P_r$ must satisfy some consistency conditions. Formulae \eqref{rhspray}
and the consistency conditions for functions
$P_1,..., P_r$  were obtained in \cite[Prop. 3.2]{CS11}. 

For homogeneous higher order differential equation fields, an
important concept is that of projective equivalence, which we borrow
from \cite[Def. 5.1]{CS11}.
\begin{defn}
Consider $S_1$ and $S_2$ two homogeneous semisprays of order $r$. We say that $S_1$ and $S_2$ are
\emph{projectively equivalent} if there exists a function $P\in
C^{\infty}(T^r_0M)$ such that $S_1=S_2-(r+1)P\mathbb{C}_r$. \end{defn}
Two homogeneous semisprays $S_1$ and $S_2$, locally given by formula \eqref{rsemispray}, are projectively equivalent if
and only if the semispray coefficients  $G^i_1$ and $G^i_2$ are related by
$G^i_1=G^i_2 + Py^{(1)i}$, for some function $P\in
C^{\infty}(T^r_0M)$. 
\begin{defn} \label{defn:rspray1} A homogeneous semispray $S\in  \mathfrak{X}(T^r_0M)$ is
  called a \emph{spray of order $r$} if $[\mathbb{C}_1,S]=S$ and
  $[\mathbb{C}_2, S]=2\mathbb{C}_1.$ \end{defn}
Above definition was proposed in \cite{CS11} for \emph{generalized
  sprays} and it is motivated by the following arguments.  It has been shown in \cite[Thm. 5.2]{CS11} that for two projectively
equivalent homogeneous semisprays their geodesics coincide up to an orientation
preserving reparameterization.  Moreover, according to
\cite[Thm. 5.2]{CS11}, the projective class of a homogeneous semispray
contains a spray, that is a homogenous semispray  for which the homogeneity conditions \eqref{rhspray} hold true with
$P_1=P_2=0$.

\section{The inverse problem of the calculus of variations for higher
  order differential equation fields} \label{sec:ls}

The inverse problem of the calculus of variations for a semispray (of
order $1$) was reformulated in \cite{BD09} in terms of semi-basic
$1$-forms. In this section we extend these aspects to the higher order
case. In Theorem \ref{thm:ls} we characterize Lagrangian semisprays of
order $2k-1$ in terms of semi-basic $1$-forms of order $k$.

\subsection{Higher order Lagrangians}

In this subsection we discuss some aspects regarding the geometry of 
a Lagrangian of order $k$. In Lemma \ref{pcel} we study these
geometric aspects in connection with the Poincar\'e-Cartan $1$-form,
which is a semi-basic $1$-form of order $k$.

Consider $L$, a Lagrangian of order $k$, which is a function defined on $T^kM$.  The \emph{Poincar\'e-Cartan
$1$-form} $\theta_L\in \Lambda^1(T^{2k-1}M)$  of $L$ is given by 
\begin{eqnarray}
\theta_L=\sum_{\alpha=1}^k\frac{(-1)^{\alpha-1}}{\alpha
  !}\mathcal{L}^{\alpha-1}_S d_{J^{\alpha}}L,
\label{tl} \end{eqnarray}
where $S\in \mathfrak{X}(T^{2k-1}M)$ is an arbitrary semispray of
order $2k-1$. We will see in Lemma \ref{pcel} that $\theta_L$ does not
depend on $S$. The \emph{Poincar\'e-Cartan $2$-form} $\omega_L\in
\Lambda^2(T^{2k-1}M)$ is given by $\omega_L=-d\theta_L$. 

The \emph{Lagrangian energy function} $\mathcal{E}_L\in
C^{\infty}(T^{2k-1}M)$ is given by 
\begin{eqnarray}
\mathcal{E}_L=\sum_{\alpha=1}^k\frac{(-1)^{\alpha-1}}{\alpha
  !}\mathcal{L}^{\alpha-1}_S \mathbb{C}_{\alpha}(L) - L.
\label{el} \end{eqnarray}
In the next Lemma we discuss some geometric aspects for a Lagrangian
$L$ of order $k$ in terms of its Poincar\'e-Cartan forms and the
Lagrangian energy function. 

\begin{lem} \label{pcel} Consider $L$ a Lagrangian of order $k$. 
\begin{itemize} \item[i)] The Poincar\'e-Cartan $1$-form $\theta_L$ is
  a semi-basic $1$-form of order $k$ on $T^{2k-1}M$, which does not
  depend on the semispray $S$. \item[ii)] The Lagrangian energy function $\mathcal{E}_L\in C^{\infty}(T^{2k-1}M)$ does not
  depend on the semispray $S$ and it is related to the
  Poincar\'e-Cartan $1$-form $\theta_L$ by the following formula
\begin{eqnarray} \mathcal{E}_L=i_S\theta_L -
  L. \label{elis} \end{eqnarray} 
\item[iii)]  The Poincar\'e-Cartan $2$-form $\omega_L$ is a symplectic
  $2$-form on $T^{2k-1}M$ if and only if the Hessian matrix \begin{eqnarray}
 g_{ij}=\frac{\partial ^2 L}{\partial y^{(k)i}\partial
   y^{(k)j}}, \label{gij} \end{eqnarray} has maximal rank $n$ on $T^kM$.
\end{itemize} 
\end{lem}
\begin{proof}
i) Locally, the Poincar\'e-Cartan $1$-form $\theta_L$ can be expressed as
follows
\begin{eqnarray}
\theta_L=\theta_{(0)i}dx^i
+ \cdots + \theta_{(k-1)i}dy^{(k-1)i}, \label{localtl2} \end{eqnarray}
where
\begin{eqnarray}
\theta_{(0)i} &=& \frac{1}{1!} \frac{\partial L}{\partial y^{(1)i}} -
\frac{1}{2!} \mathcal{L}_S\left( \frac{\partial L}{\partial
    y^{(2)i}}\right) + \cdots + \frac{(-1)^{k-1}}{k!} \mathcal{L}_S^{k-1}\left( \frac{\partial L}{\partial
    y^{(k)i}}\right), \nonumber \\
\theta_{(1)i} &=& \frac{1}{2} \frac{\partial L}{\partial y^{(2)i}} -
\frac{1}{2\cdot 3} \mathcal{L}_S \left( \frac{\partial L}{\partial
    y^{(3)i}}\right) + \cdots + \frac{(-1)^{k-2}}{2\cdot 3 \cdots k}
  \mathcal{L}_S^{k-2}\left( \frac{\partial L}{\partial y^{(k)i}}\right), \nonumber \\
&\ldots & \label{theta0k} \\
\theta_{(k-2)i} & = & \frac{1}{k-1} \frac{\partial L}{\partial y^{(k-1)i}} -
\frac{1}{k(k-1)}\mathcal{L}_S\left( \frac{\partial L}{\partial
    y^{(k)i}}\right), \nonumber \\
\theta_{(k-1)i}  & = & \frac{1}{k} \frac{\partial L}{\partial y^{(k)i}}. \nonumber
\end{eqnarray}
Consider $d_T$, the Tulczyjew operator \eqref{dtr} on $T^{2k-1}M$. $L$ is a
Lagrangian on $T^kM$ and ${\partial L}/{\partial
  y^{(\alpha)i}}$ are locally defined on $T^kM$, for all
  $\alpha\in\{1,..., k\}$. Therefore, we can view 
\begin{eqnarray}
\mathcal{L}^{\beta}_S\left(\frac{\partial L}{\partial
  y^{(\alpha)i}}\right)=d_T^{\beta}\left(\frac{\partial L}{\partial
  y^{(\alpha)i}}\right), \label{lsbla} 
\end{eqnarray}
as locally defined functions on $T^{k+\beta}M,$ for all $\beta \in
\{1,..., k-1\}$. It follows that all components $\theta_{(\alpha)i}$, $\alpha
\in \{0,...,k-1\}$, in formula \eqref{theta0k}, do not depend on the
semispray $S$. From formula \eqref{localtl2} it follows that
$\theta_L$ is a semi-basic $1$-form of order $k$, which does not depend on the semispray $S$. 

ii) Since for all $\alpha \in \{1,..., k\}$ the functions
$\mathbb{C}_{\alpha}(L)$ are defined on $T^kM$, it follows that we can
view the functions
$$ \mathcal{L}^{\alpha-1}_S \mathbb{C}_{\alpha}(L) = d_T^{\alpha-1}
\mathbb{C}_{\alpha}(L)  $$ as being defined on
$T^{k+\alpha-1}M$. Therefore, the right hand side of formula
\eqref{el}, and hence the energy $\mathcal{E}_L$, is independent of
the choice of the semispray $S$.  

If we apply $i_S$ to both sides of formula \eqref{tl} it follows 
\begin{eqnarray}
\nonumber i_S\theta_L &=&\sum_{\alpha=1}^k\frac{(-1)^{\alpha-1}}{\alpha
  !} i_S\mathcal{L}^{\alpha-1}_S d_{J^{\alpha}}L = \sum_{\alpha=1}^k\frac{(-1)^{\alpha-1}}{\alpha
  !} \mathcal{L}^{\alpha-1}_S i_S d_{J^{\alpha}}L \\ &=& \sum_{\alpha=1}^k\frac{(-1)^{\alpha-1}}{\alpha
  !} \mathcal{L}^{\alpha-1}_S \mathcal{L}_{J^{\alpha}S}L = \sum_{\alpha=1}^k\frac{(-1)^{\alpha-1}}{\alpha
  !} \mathcal{L}^{\alpha-1}_S \mathbb{C}_{\alpha}L.
\label{istl} \end{eqnarray}
In the above formula we did use the commutation rule
$i_Sd_{J^{\alpha}}+ d_{J^{\alpha}} i_S = \mathcal{L}_{J^{\alpha}S}
+i_{[J^{\alpha}, S]}$, \cite[A.1]{GM00}, as well as the fact that
  $J^{\alpha}S=C_{\alpha}$. From formula \eqref{istl} we obtain that
  \eqref{elis} is true. 

iii) Using formula \eqref{localtl2} and the fact that we can view
$\theta_{(k-\alpha)i}$ as locally defined functions on
$T^{(k+\alpha-1)}M$, it follows that 
\begin{eqnarray} 2kn \geq \textrm{rank} (d\theta_L) \geq 2\cdot \sum_{\alpha=1}^k
  \textrm{rank} \left( \frac{\partial \theta_{(k-\alpha)i}}{\partial
      y^{(k+\alpha-1)j}}\right). \label{rankol} \end{eqnarray} 
Since ${\partial L}/{\partial y^{(k)i}}$ are locally defined functions on
  $T^kM$, we have 
\begin{eqnarray} \mathcal{L}_S^{\alpha}\left(\frac{\partial
      L}{\partial y^{(k)i}}\right)= (k+1)\cdots
  (k+\alpha)y^{(k+\alpha)j} g_{ij} + f_{\alpha}, \label{lsalk}\end{eqnarray}
for $f_{\alpha}$ locally defined functions on $T^{k+\alpha-1}M$. Using
the formulae \eqref{lsbla} and \eqref{lsalk} and the 
components $\theta_{(\alpha)i}$ of the Poincar\'e-Cartan $1$-form $\theta_L$ it
follows
\begin{eqnarray} \frac{\partial \theta_{(k-\alpha)i}}{\partial
    y^{(k+\alpha-1)j}} = (-1)^{\alpha-1}
\frac{(k+\alpha-1)!(k-\alpha)!}{(k!)^2} g_{ij}, \forall \alpha \in
\{1,...,k\}. \label{pty}\end{eqnarray}  
Now, from formulae \eqref{rankol} and \eqref{pty} it follows that 
\begin{eqnarray} 2kn \geq \textrm{rank} (d\theta_L) \geq 2k \cdot
  \textrm{rank}(g_{ij}). \label{rankolg} \end{eqnarray}
We prove the first implication of part iii) of the lemma by
contradiction. We assume that $\omega_L=-d\theta_L$ is a symplectic
structure on $T^{2k-1}M$ and also that $\textrm{rank}(g_{ij})<n$. It
  follows that there are locally defined functions $X^i$ such that
  $g_{ij}X^j=0$. It follows that the non-zero vector field
  $X=X^i{\partial}/{\partial y^{(2k-1)i}}$ satisfies
  $i_Xd\theta_L=0$, which contradicts the fact that $\omega_L$ is a
  symplectic structure. The converse implication of the third item of
  the lemma follows directly from formula \eqref{rankolg}. If
  $\textrm{rank}(g_{ij})=n$ we obtain that
  $\textrm{rank}(d\theta_L)=2n$ and hence $\omega_L$ is a symplectic structure. 
\end{proof}

The components $\theta_{(\alpha)i}$, $\alpha\in \{0,...,k-1\}$, of
the Poincar\'e-Cartan $1$-form $\theta_L$, in formula \eqref{theta0k},
are the Jacobi-Ostrogradski generalized momenta, \cite{deleon95}. 

The local expression \eqref{localtl2} - \eqref{theta0k} for $\theta_L$ can be written in a more compact form as follows
\begin{eqnarray}
\theta_L=\sum_{\alpha=1}^k(\alpha-1)! \left\{\sum_{\beta=\alpha}^{k}
  \frac{(-1)^{\beta-\alpha}}{\beta
    !}\mathcal{L}_S^{\beta-\alpha}\left(\frac{\partial L}{\partial
      y^{(\beta)i}}\right)\right\}
dy^{(\alpha-1)i}. \label{localtl1}\end{eqnarray}

\begin{defn} \label{regularlk}  A Lagrangian $L$ of order $k$ is said
  to de regular if the Poincar\'e-Cartan $2$-form $\omega_L$ is a
  symplectic $2$-form on $T^{2k-1}M$. \end{defn}
Using part iii) of Lemma \ref{pcel}  we have that a Lagrangian $L$ of order
$k$ is regular if and only if the
 Hessian matrix \eqref{gij} has maximal rank $n$ on
$T^kM$. These regularity conditions correspond to the regularity
conditions for minimal-order Lagrangians proposed by O. Krupkov\'a in
\cite[Chapter 6]{krupkova97}.

\subsection{Lagrangian semisprays}  \label{subsec:ls} 

The inverse problem of the calculus of variations for systems of
higher order ordinary differential equations can be formulated as
follows.  Under what conditions the solutions of the  system
\eqref{rode} of order $2k$ coincide with the solutions of the
Euler-Lagrange equations
\begin{eqnarray}
\frac{\partial L}{\partial x^i} - \frac{1}{1!} \frac{d}{dt}\left(\frac{\partial
    L}{\partial y^{(1)i}}\right) + \cdots + \frac{(-1)^k}{k!} \frac{d^k}{dt^k}\left(\frac{\partial
    L}{\partial y^{(k)i}}\right) =0, \label{elk}
\end{eqnarray} for some Lagrangian $L$ of order $k$?
The equivalence of the two systems \eqref{rode} and \eqref{elk}
require that the Hessian matrix \eqref{gij}, of the sought after Lagrangian $L$ of order
$k$, has rank $n$ and hence the Lagrangian has to be regular. 

\begin{defn} \label{lsk} A semispray $S$, of order
  $2k-1$, is called a \emph{Lagrangian semispray} if its geodesics,
  which are solutions to the system \eqref{rode}, for $r=2k-1$, are 
  solutions to the Euler-Lagrange equations \eqref{elk}, for some
 regular Lagrangian $L$ of order $k$, defined locally on some open
 domain in $T^kM$. \end{defn} 

For a given semispray of order $2k-1$, the Lagrangian to search for can
be of order higher then $k$ and the regularity condition can be more
general, see \cite{krupkova86, krupkova97}. In this work, we focus our
attention on Lagrangians of minimal-order and hence the regularity
condition is given in Definition \ref{regularlk}

Next theorem provides a characterization for Lagrangian semisprays,
in terms of semi-basic $1$-forms, extending the results obtained in
\cite{BD09}. In \cite[Thm. 3.2]{deLeon92}, Lagrangian
semisprays of order $2k-1$ are characterized in terms of a closed
$2$-form, extending the $k=1$ case, which was studied in
\cite{crampin81}. The relationship between variational equations of an
arbitrary order and closed $2$-forms has been investigated in
\cite{krupkova86, krupkova87}. 

\begin{thm} \label{thm:ls} Consider $S$ a semispray of order
  $2k-1$. \begin{itemize} \item[i)] $S$ is a Lagrangian semispray if
    and only if there exists a (locally defined) regular Lagrangian $L$ of order $k$ such
    that either one, of the following equivalent two conditions, is
    satisfied
\begin{eqnarray} \mathcal{L}_S\theta_L=dL, \quad
  i_S\omega_L=d\mathcal{E}_L. \label{isoel} \end{eqnarray}
\item[ii)]  $S$ is a Lagrangian semispray if
    and only if there exists a (locally defined) semi-basic $1$-form $\theta$ on
    $T^{2k-1}M$ of order $k$ such that
    $\textrm{rank}(d\theta) = 2kn$ and the $1$-form
    $\mathcal{L}_S\theta$ is closed. In this case $\theta$ is the
    Poincar\'e-Cartan $1$-form of some locally defined regular Lagrangian $L$
    of order $k$. \end{itemize} \end{thm}
\begin{proof}
i) Using the Euler-Lagrange equations \eqref{elk}, it follows that the semispray $S$ is Lagrangian if and only if it satisfies the equation
\begin{eqnarray}
\frac{\partial L}{\partial x^i} - \frac{1}{1!} S\left(\frac{\partial
    L}{\partial y^{(1)i}}\right) + \cdots + \frac{(-1)^k}{k!} S^k\left(\frac{\partial
    L}{\partial y^{(k)i}}\right) =0, \label{els}
\end{eqnarray} for some (locally defined) regular Lagrangian $L$ of
order $k$. 

In view of formula \eqref{elis} we obtain that the two
equations \eqref{isoel} are equivalent. Therefore,  we will have to 
prove that equation \eqref{els} and first equation \eqref{isoel}
are equivalent.

Using expression \eqref{localtl1} for the Poincar\'e-Cartan
$1$-form $\theta_L$ and the fact that
$\mathcal{L}_Sdy^{(\alpha-1)i}=\alpha dy^{(\alpha)i}$ it follows 
\begin{eqnarray*}
\nonumber \mathcal{L}_S\theta_L & = & \sum_{\alpha=1}^k(\alpha-1)! \left\{\sum_{\beta=\alpha}^{k}
  \frac{(-1)^{\beta-\alpha}}{\beta
    !}\mathcal{L}_S^{\beta-\alpha+1}\left(\frac{\partial L}{\partial
      y^{(\beta)i}}\right)\right\} dy^{(\alpha-1)i} \\
 \nonumber &  & + \sum_{\alpha=1}^k\alpha ! \left\{\sum_{\beta=\alpha}^{k}
  \frac{(-1)^{\beta-\alpha}}{\beta !}\mathcal{L}_S^{\beta-\alpha}\left(\frac{\partial L}{\partial
      y^{(\beta)i}}\right)\right\} dy^{(\alpha)i} \\
& = & \sum_{\beta=1}^k \frac{(-1)^{\beta-1}}{\beta
  !}\mathcal{L}_S^{\beta}\left(\frac{\partial L}{\partial
    y^{(\beta)i}}\right) dx^i + \sum_{\alpha=1}^k \frac{\partial L}{\partial
    y^{(\alpha)i}} dy^{(\alpha)i} . \label{lstl}\end{eqnarray*}
If we use the above expression for $\mathcal{L}_S\theta_L$ it
follows that
\begin{eqnarray} dL- \mathcal{L}_S\theta_L =\left\{ \frac{\partial
      L}{\partial x^i} + \sum_{\beta=1}^k \frac{(-1)^{\beta}}{\beta
  !}\mathcal{L}_S^{\beta}\left(\frac{\partial L}{\partial
    y^{(\beta)i}}\right) \right\} dx^i \label{dllstl} \end{eqnarray}
is a semi-basic $1$-form on $T^{2k-1}M$ of order $1$. Formula
\eqref{dllstl} shows that equation \eqref{els} and first equation \eqref{isoel}
are equivalent. 

ii) For the direct implication of this part, we assume that $S$ is a
Lagrangian semispray. Therefore, semispray $S$ satisfies first equation
\eqref{isoel}, for some regular Lagrangian $L$ of
order $k$. We consider
$\theta=\theta_L\in \Lambda^{1}(T^{2k-1}M)$, its Poincar\'e-Cartan
$1$-form, which is a semi-basic $1$-form of order $k$ and satisfies
first equation \eqref{isoel}. By Definition \ref{regularlk} we have
that $\textrm{rank}(d\theta)=2kn$.

For the converse, let us consider $\theta\in \Lambda^{1}(T^{2k-1}M)$, a
semi-basic $1$-form of order $k$ such that $\mathcal{L}_S\theta$ is a
closed $1$-form. Therefore $\mathcal{L}_S\theta$ is locally exact and
hence there exists $L$, a locally defined function on $T^{2k-1}M$, such
that 
\begin{eqnarray} \mathcal{L}_S\theta =
  dL. \label{lstdl} \end{eqnarray}
We want to prove now that $L$ is constant on the fibres $\pi^{2k-1}_k$
and hence we can view it as a function defined on some open domain of $T^kM$. Moreover, we will
prove that $\theta$ is the Poincar\'e-Cartan $1$-form
$\theta_L$ of $L$. For these, as we have seen in the last part of
Lemma \ref{lem:cas}, we need a condition weaker then
\eqref{lstdl}, namely we will use the fact that
$\mathcal{L}_S\theta-dL$ is a semi-basic $1$-form of order $1$. This
means that 
\begin{eqnarray*} i_J\mathcal{L}_S\theta =
  d_JL. \label{ijlstdl} \end{eqnarray*}
According to part v) of Lemma \ref{lem:cas} it follows that one can
restrict the function $L$ to some open domain of $T^kM$ and the semi-basic $1$-form $\theta$
is given by formula \eqref{tab}, where $f=L$ and $\alpha=k$. It
follows that $\theta$ is given by
formula \eqref{tl} and hence it is the Poincar\'e-Cartan $1$-form of
the function $L$, which means that
$\theta=\theta_L$. Using the assumption
$\textrm{rank}(d\theta) = 2kn$ it follows that the Poincar\'e-Cartan $2$-form of $L$,
$\omega_L=-d\theta_L=-d\theta$, is a symplectic structure. Hence $L$ is a (locally defined)
regular Lagrangian of order $k$. If we replace $\theta=\theta_L$ in formula
\eqref{lstdl} it follows that the semispray $S$ satisfies first
formula \eqref{isoel} for the Lagrangian $L$. In view
of the first part of the theorem it follows that the semispray $S$ is Lagrangian.
\end{proof}

According to Definition \ref{regularlk}, we have that for a regular
Lagrangian $L$ of order $k$, second equation \eqref{isoel} has a unique
solution. This way, to each regular Lagrangian $L$ on $T^kM$ it
corresponds a unique Lagrangian semispray $S\in
\mathfrak{X}(T^{2k-1}M)$. We will refer to this semispray as to the
canonical semispray (or the Euler-Lagrange vector field) associated to the Lagrangian $L$ of order $k$. Using the terminology introduced by Krupkov\'a
in \cite[Ch. 4]{krupkova97} we can say that for a regular Lagrangian its
Euler-Lagrange distribution has a constant rank equal to one and it is
spanned by the semispray $S$.    

If we want to determine the local coefficients $G^i$ of a Lagrangian semispray $S$
of order $2k-1$, we use formula \eqref{lsalk} and write equations
\eqref{els} in the following equivalent form 
\begin{eqnarray}
(-1)^k {2k \choose k} g_{ij}G^j =  \frac{\partial L}{\partial x^i} - \frac{1}{1!} d_T\left(\frac{\partial
    L}{\partial y^{(1)i}}\right) + \cdots + \frac{(-1)^k}{k!} d^k_T\left(\frac{\partial
    L}{\partial y^{(k)i}}\right). \label{gijgj} 
\end{eqnarray} It follows that for a regular Lagrangian, the Hessian matrix
$g_{ij}$ is invertible and hence equations \eqref{gijgj} uniquely
determine the semispray coefficients $G^i$. 

For a Lagrangian semispray $S$, its geodesics are solutions of the
Euler-Lagrange equations \eqref{elk}. Moreover, the geodesic equations
\eqref{rode}, with $r=2k-1$, for the Lagrangian semispray $S$ and the
Euler-Lagrange equations \eqref{elk} are related by
\begin{eqnarray}
\frac{(-1)^k}{k!}g_{ij}\left\{\frac{d^{2k}x^j}{dt^{2k}} + (2k)! G^j \right\} =  \frac{\partial L}{\partial x^i} - \frac{1}{1!} \frac{d}{dt}\left(\frac{\partial
    L}{\partial y^{(1)i}}\right) + \cdots + \frac{(-1)^k}{k!} \frac{d^k}{dt^k}\left(\frac{\partial
    L}{\partial y^{(k)i}}\right) \end{eqnarray}
where $g_{ij}$ is the Hessian matrix \eqref{gij}. The two systems of
equations \eqref{rode} and \eqref{elk} coincide if the Lagrangian is
regular.

Next lemma presents some compatibility conditions between the
geometric structures associated to a Lagrangian and the Liouville
vector fields. These properties will be useful in the next section to
see how the homogeneity properties of a Finsler function transfer to
the induced geometric structures. 
\begin{lem} \label{lem:ictl}
Consider $L$ a Lagrangian on $T^kM$ and $\theta_L\in \Lambda^1(T^{2k-1}M)$ its
Poincar\'e-Cartan $1$-form. The following formulae are true:
\begin{eqnarray}
\nonumber \mathcal{L}_{\mathbb{C}_1}\theta_L & = & \theta_{\mathbb{C}_1(L)-L}, \\
i_{\mathbb{C}_{\alpha}}\theta_L &=& \alpha !
\sum_{\beta=1}^{k-\alpha}\frac{(-1)^{\beta-1}}{(\alpha+\beta)!}
\mathcal{L}_S^{\beta-1}\left(\mathbb{C}_{\alpha +
    \beta}(L)\right), \quad \forall \alpha \in \{1, ...,
k-1\}, \label{ictl} \\
\nonumber i_{\mathbb{C}_{\alpha}}\theta_L &=& 0, \quad \forall \alpha \in \{k, ...,
2k-1\}.
\end{eqnarray}
\end{lem}
\begin{proof}
For the Lagrangian function $L$ consider $S$ a semispray, solution to
one of the two equivalent equations \eqref{isoel}, which means
$\mathcal{L}_S\theta_L=dL$. If we apply $\mathcal{L}_{\mathbb{C}_1}$ to
both sides of this formula and use the commutation rule we obtain 
\begin{eqnarray} 
\mathcal{L}_S\mathcal{L}_{\mathbb{C}_1}\theta_L +
\mathcal{L}_{[\mathbb{C}_1, S]}\theta_L = d\mathbb{C}_1(L). \label{lslc1t1}\end{eqnarray}
Using formula \eqref{cas}, it follows that $[\mathbb{C}_1, S]=S+U_1$,
for $U_1\in \mathfrak{X}^{V_{2k-1}}(T^{2k-1}M)$. If we replace $[\mathbb{C}_1,
S]$ in formula \eqref{lslc1t1} we obtain  
\begin{eqnarray} 
\mathcal{L}_S\mathcal{L}_{\mathbb{C}_1}\theta_L =
d\left(\mathbb{C}_1(L)-L\right) - \mathcal{L}_{U_1}\theta_L. \label{lslc1t2}\end{eqnarray}
Using the local expression \eqref{localtl2} of the Poincar\'e-Cartan
$1$-form $\theta_L$ and the fact that its only component that depends
on $y^{(2k-1)i}$ is $\theta_{(0)i}$, which is given in formula
\eqref{theta0k}, it follows that 
\begin{eqnarray*}
\mathcal{L}_{U_1}\theta_L=\mathcal{L}_{U_1}\left(\theta_{(0)i}\right)
dx^i. \end{eqnarray*}
Therefore $\mathcal{L}_{U_1}\theta_L$ is a semi-basic $1$-form of
order $1$. Using formula \eqref{lslc1t2} it follows  
\begin{eqnarray} 
i_J\mathcal{L}_S\mathcal{L}_{\mathbb{C}_1}\theta_L =
i_Jd\left(\mathbb{C}_1(L)-L\right). \label{lslc1t3}\end{eqnarray}
According to part iv) of Lemma \ref{lem:cas} it follows
that $\mathcal{L}_{\mathbb{C}_1}\theta_L$ is a semi-basic $1$-form of
order $k$. We use now part v) of Lemma \ref{lem:cas} to conclude,
from formula \eqref{lslc1t3}, that the semi-basic $1$-form of order $k$,
$\mathcal{L}_{\mathbb{C}_1}\theta_L$, satisfies formula \eqref{tab}
for $\alpha=k$ and $f=\mathbb{C}_1(L)-L $. In view of formula
\eqref{tl}, this means that $\mathcal{L}_{\mathbb{C}_1}\theta_L$  is the Poincar\'e-Cartan $1$-form
of the function $\mathbb{C}_1(L)-L$, which is first formula \eqref{ictl}.

Now, we use formula $\mathcal{L}_S\theta_L=dL$ and compose both sides
with $i_J$, which means that $ i_J\mathcal{L}_S\theta_L=d_JL$.
Using this formula and part v) of Lemma \ref{lem:cas} 
it follows that the semi-basic $1$-form of order $k$,  $\theta_L$
satisfies formulae \eqref{tabg} for $\alpha=k$, $\gamma \in \{0,...,
k-1\}$ and $f=L$, which can be written as follows
 \begin{eqnarray}
i_{J^{\alpha}}\theta_L &=& \alpha !
\sum_{\beta=1}^{k-\alpha}\frac{(-1)^{\beta-1}}{(\alpha+\beta)!}
\mathcal{L}_S^{\beta-1} d_{J^{\alpha +
    \beta}}L, \quad \forall \alpha \in \{1, ...,
k-1\}. \label{ijatl} \end{eqnarray}
We note that both sides in above formulae do not depend on the chosen
semispray $S$. If we compose with $i_S$ in both sides of formulae
\eqref{ijatl}, we obtain formulae \eqref{ictl} for $\alpha \in \{1, ...,
k-1\}$. 

Since $\theta_{L}$ is a semi-basic $1$-form of order $k$, it follows
that there exists $\eta\in \Lambda^1(T^{2k-1}M)$ such that
$\theta_L=i_{J^k}\eta$. For $\alpha \in \{k, ..., 2k-1\}$, we have
that $J^k(\mathbb{C}_{\alpha})=0$. Therefore,
$ i_{\mathbb{C}_{\alpha}}\theta_L=i_{J^k(\mathbb{C}_{\alpha})}\eta=0$
and hence we proved all formulae \eqref{ictl} 
\end{proof}
The $1$-forms $i_{J^{\alpha}}\theta_L\in \Lambda^1(T^{2k-1}M)$,
$\alpha\in \{0,...,k-1\}$ are semi-basic $1$-forms of order
$k-\alpha$.

We prove in the next proposition that some
homogeneity properties of a regular Lagrangian are inherited by its
canonical semispray. 

\begin{prop} \label{prop:homls} Consider $L$ a regular Lagrangian
  of order $k$ such that $\mathbb{C}_1(L)=aL$, for $a\neq 1$, and let
  $S$ be its canonical semispray of order $2k-1$. It follows that
  $[\mathbb{C}_1, S]=S$. \end{prop}
\begin{proof}
Since $L$ is a regular Lagrangian of order $k$ it follows that the semispray $S\in \mathfrak{X}(T^{2k-1}_0M)$ is the unique solution of
the second equation \eqref{isoel}. Using the fact that
$\mathcal{L}_S\omega_L =0$, it follows that 
\begin{eqnarray}
i_{[\mathbb{C}_1, S]}\omega_L  = i_{\mathbb{C}_1}\mathcal{L}_S\omega_L -
\mathcal{L}_S i_{\mathbb{C}_1} \omega_L = \mathcal{L}_S
i_{\mathbb{C}_1} d\theta_L = \mathcal{L}_S
\left(\mathcal{L}_{\mathbb{C}_1}\theta_L - di_{\mathbb{C}_1}
  \theta_L\right) . \label{ic1so1}
\end{eqnarray}
If we use first formula \eqref{ictl} and the homogeneity condition
$\mathbb{C}_1(L)=aL$ we obtain $\mathcal{L}_{\mathbb{C}_1}\theta_L
=(a-1)\theta_L$. We replace this and first formula \eqref{isoel} in
\eqref{ic1so1}. It follows 
\begin{eqnarray}
i_{[\mathbb{C}_1, S]}\omega_L  = (a-1)\mathcal{L}_S
\theta_L - d\mathcal{L}_Si_{\mathbb{C}_1}
  \theta_L = (a-1)dL - d\mathcal{L}_Si_{\mathbb{C}_1}
  \theta_L  . \label{ic1so2}
\end{eqnarray}
Using second formula \eqref{ictl}, for $\alpha=1$, we obtain the
following expression for the energy Lagrangian function
$\mathcal{E}_L$, which is given by formula \eqref{el}
\begin{eqnarray}
\mathcal{E}_L=\mathbb{C}_1(L)-\mathcal{L}_Si_{\mathbb{C}_1}\theta_L -
L= (a-1)L -
\mathcal{L}_Si_{\mathbb{C}_1}\theta_L. \label{eal}\end{eqnarray} 
We replace the expression for $\mathcal{L}_Si_{\mathbb{C}_1}\theta_L$
from above formula in \eqref{ic1so2} and obtain 
\begin{eqnarray*}
i_{[\mathbb{C}_1, S]}\omega_L = d\mathcal{E}_L= i_{
  S}\omega_L. \end{eqnarray*}
Since $\omega_L$ is a symplectic structure it follows that $[\mathbb{C}_1, S]=S$. 
\end{proof}

For the case $k=1$, above formulae show that the homogeneity of a
regular Lagrangian transfers to the canonical Euler-Lagrange vector
field, which makes it into a spray. 

\section{Projective metrizability for homogeneous higher order
  differential equation fields} \label{sec:pm}

A particular aspect of the inverse problem of the calculus of
variations deals with homogeneous systems of differential
equations. For $k=1$, this problem is known as the projective
metrizability problem, or as the Finslerian version of Hilbert's
fourth problem \cite{alvarez05, crampin08, cms12, saunders12, szilasi07}. The most
important aspect that is needed to formulate and address the
projective metrizability problem for $k>1$ relies on a correct
definition of homogeneity for systems of higher order differential equations and corresponding Lagrangians. We believe
that such definition of homogeneity is that proposed by Crampin and
Saunders in \cite{CS11}, which we use in this paper. In this section
we formulate and discuss some aspects regarding the projective
metrizability problem for the case $k>1$, extending some results 
obtained in \cite{BD09, BM11a} for $k=1$. 

\subsection{Higher order Finsler spaces} \label{subsec:hofs}

For $k=1$, a Finsler function is characterized by the following important
aspect: its variational problem uniquely determines a class of
projectively related systems of second order ordinary differential
equations. This property is due to  the fact that a Finsler
function satisfies some homogeneity condition and a regularity
condition. Inspired by the work of Crampin and Saunders \cite{CS11}, we propose
the following definition for a Finsler function of order $k>1$.

\begin{defn} \label{finslerk}
A positive function $F\in C^{\infty}(T^k_0M)$ is called a \emph{Finsler function of
order $k$} if 
\begin{itemize} \item[i)] it satisfies the Zermelo conditions:
\begin{eqnarray}  \mathbb{C}_1(F)=F, \quad \mathbb{C}_{\alpha}(F)=0, \
  \forall \alpha \in\{2,...,k\}, \label{zermelo} \end{eqnarray}
\item[ii)] the tensor with components
\begin{eqnarray}
h_{ij}=F^{2k-1}\frac{\partial^2 F}{\partial y^{(k)i}\partial
  y^{(k)j}} \label{hij} \end{eqnarray}
has rank $n-1$ on $T^k_0M$.
\end{itemize}
\end{defn}
A Lagrangian $L$ on $T^k_0M$ that satisfies the Zermelo conditions \eqref{zermelo}
in Definition \ref{finslerk} is called \emph{parametric Lagrangian} in
\cite[\S 4]{CS11} since the solutions of the corresponding variational
problem are invariant under orientation preserving
reparameterization. The Zermelo conditions and the invariance under
reparameterizations for the integral curves of some higher order
differential equations, as well as  their relation with the variational
equations related to Finsler geometry, has been studied very recently
by Urban and Krupka in \cite{uk13}. 

Spaces with functions that satisfy the Zermelo conditions \eqref{zermelo} as well
as the regularity condition ii) of Definition \ref{finslerk} where
studied by Kawaguchi, \cite{kawaguchi62}, and also referred to as
\emph{Kawaguchi spaces}.

Definition \ref{finslerk} reduces to the classic definition of a
Finsler space when $k=1$, and the tensor \eqref{hij} becomes the
angular metric tensor \cite[\S 16]{matsumoto86}. Indeed, if $k=1$, we
have that the tensor \eqref{hij} satisfies
\begin{eqnarray*} 
h_{ij} = \frac{1}{2}\frac{\partial^2F^2}{\partial y^i\partial y^j} -
\frac{\partial F}{\partial y^i} \frac{\partial F}{\partial
  y^j}. \label{hgij} \end{eqnarray*}
It is well known that $\textrm{rank}(h_{ij})=n-1$ if and only if
$\textrm{rank}({\partial^2F^2}/{\partial y^i\partial y^j})=n$,
\cite[\S 16]{matsumoto86}. Due to a recent result by Lovas
\cite{lovas07}, the regularity condition
$\textrm{rank}({\partial^2F^2}/{\partial y^i\partial y^j})=n$ and the
positivity of the Finsler function $F$ is equivalent to the fact that
Hessian matrix of $F^2$,  $g_{ij}={\partial^2F^2}/{\partial y^i\partial y^j}$
is positive definite. Using \cite[Section 3]{cms12} or \cite[Section
3]{saunders12} the Hessian matrix of $F^2$ is positive definite if and
only if the Hessian matrix of $F$ is positive quasi-definite. 

\begin{defn} \label{def:homtheta}
A $1$-form $\theta \in \Lambda^1(T^{2k-1}_0M)$ is called
\emph{homogeneous} if it satisfies the formulae
\begin{eqnarray}
\label{icat} i_{\mathbb{C}_{\alpha}}\theta & =& 0, \quad
\mathcal{L}_{\mathbb{C}_{\alpha}}\theta=0, \quad  \forall
\alpha\in \{1,..., 2k-1\}.
\end{eqnarray}
\end{defn}

Due to the homogeneity conditions of a Finsler function of order $k$,
the energy function $\mathcal{E}_F$ and the
Poincar\'e-Cartan forms $\theta_F$ and $\omega_F=-d\theta_F$ have
special properties. These properties are presented in the next lemma. 

Last part of the next lemma also shows that the regularity condition ii) in
Definition \ref{finslerk} is equivalent to
$\operatorname{rank}(d\theta_F) =2k(n-1)$, which is the regularity
condition for parametric Lagrangians considered by Crampin and
Saunders in \cite{CS11}. 

\begin{lem} \label{lem:fthom}
Consider $F\in C^{\infty}(T^k_0M)$ a positive function that satisfies the
Zermelo conditions \eqref{zermelo} and $S\in
\mathfrak{X}(T^{2k-1}_0M)$ a semispray, solution of the equation
$\mathcal{L}_S\theta_F=dF$. 
\begin{itemize} \item[i)] The  Poincar\'e-Cartan
$1$-form $\theta_F$ satisfies the homogeneity conditions \eqref{icat}.
\item[ii)] The following formulae are true
\begin{eqnarray}
\label{istf} i_S\theta_F & = & F, \quad
\mathcal{E}_F=0.  \\
\label{icaof} i_{\mathbb{C}_{\alpha}}\omega_F &=& 0, \  \forall
\alpha\in \{1,..., 2k-1\}.
\end{eqnarray}
\item[iii)] $F$ is a Finsler function of order $k$ if and only if
  $\operatorname{rank}(d\theta_F) =2k(n-1)$. 
\end{itemize}
\end{lem}
\begin{proof}
Using second formulae \eqref{ictl} it follows
$i_{\mathbb{C}_{\alpha}}\theta_F=0$ if and only if
$\mathbb{C}_{\alpha+1}(F)=0$ for all $\alpha \in \{1,..., k-1\}$. 

Since $F$ satisfies the Zermelo conditions \eqref{zermelo} it follows
$i_{\mathbb{C}_{\alpha}}\theta_F=0$  for all $\alpha \in \{1,...,
k-1\}$. Last formulae \eqref{ictl} show that $i_{\mathbb{C}_{\alpha}}\theta_F=0$  for all $\alpha \in \{k,...,
2k-1\}$.

First formula \eqref{ictl} shows that the Zermelo condition
$\mathbb{C}_1(F)=F$ implies $\mathcal{L}_{\mathbb{C}_1}\theta_F=0$. 

$S\in \mathfrak{X}(T^{2k-1}_0M)$ is a semispray and satisfies the equation
$\mathcal{L}_S\theta_F=dF$. For $\alpha\geq 2$, we apply
$\mathcal{L}_{\mathbb{C}_{\alpha}}$ to both sides of this equation. It follows
\begin{eqnarray}
\mathcal{L}_S\mathcal{L}_{\mathbb{C}_{\alpha}}\theta_F+
\mathcal{L}_{[\mathbb{C}_{\alpha}, S]}\theta_F =
d\mathbb{C}_{\alpha}(F)=0. \label{lslcat1} \end{eqnarray}
Using formula \eqref{cas}, for each $\alpha \in \{2,..., 2k-1\}$ there
exists $U_{\alpha} \in \mathfrak{X}^{V_{2k-1}}(T^{2k-1}_0M)$ such that
$[\mathbb{C}_{\alpha}, S]=\alpha \mathbb{C}_{\alpha-1} + U_{\alpha}$.
We replace this in formula \eqref{lslcat1} and obtain
\begin{eqnarray}
\mathcal{L}_S\mathcal{L}_{\mathbb{C}_{\alpha}}\theta_F+
\alpha\mathcal{L}_{\mathbb{C}_{\alpha-1}}\theta_F =
-\mathcal{L}_{U_{\alpha}}\theta_F, \forall \alpha \in \{2,..., 2k-1\}. \label{lslcat2} \end{eqnarray}
Using a similar argument that we have used in the proof of Lemma
\ref{lem:ictl} it follows that $\mathcal{L}_{U_{\alpha}}\theta_F$ are
semi-basic $1$-forms of order $1$, for all $\alpha \in \{2,...,
2k-1\}$. For $\alpha=2$, in formula \eqref{lslcat2}, it follows
that $\mathcal{L}_S\mathcal{L}_{\mathbb{C}_{2}}\theta_F$ is a
semi-basic $1$-form of order $1$. Item v) of Lemma \ref{lem:cas}
implies $\mathcal{L}_{\mathbb{C}_{2}}\theta_F=0$. We continue with
$\alpha \in \{3,..., 2k-1\}$ in formula \eqref{lslcat2}, use a
similar argument as above, and obtain $\mathcal{L}_{\mathbb{C}_{\alpha}}\theta_F=0$.

For $\theta_F$, the Poincar\'e-Cartan $1$-form of a Finsler
function $F$, given by formula \eqref{tl}, we use formula \eqref{istl},
as well as the Zermelo conditions \eqref{zermelo}, to obtain
$i_S\theta_F=\mathbb{C}_1(F)=F$, which is first formula
\eqref{istf}. These considerations and formula \eqref{elis} imply that
second formula \eqref{istf} is true. 

The Poincar\'e-Cartan $1$-form is homogeneous, which means that it
satisfies formulae \eqref{icat}. The two formulae \eqref{icat} imply that formulae
\eqref{icaof} are true as well.

iii) We have seen already that the $\{S, \mathbb{C}_1, ...,
\mathbb{C}_{2k-1}\}\subset \operatorname{Ker}{\omega_F}$. Based on
this aspect and using a similar argument we did use for the proof of third
item in Lemma \ref{pcel}, formula \eqref{rankol} has the following
correspondent
\begin{eqnarray}
2k(n-1)\geq \textrm{rank}(\omega_F) \geq 2k \cdot 
\textrm{rank}(h_{ij}). \label{rankof} \end{eqnarray}   
We assume now that $F$ is a Finsler function of order $k$, which means that it
satisfies the regularity condition ii) of Definition
\ref{finslerk}. From formula \eqref{rankof} it follows that
$\textrm{rank}(\omega_F)=2k(n-1)$, which is the regularity condition
for parametric Lagrangians in \cite{CS11}.

We prove the other implication by contradiction. We assume that
$\textrm{rank}(\omega_F)=2k(n-1)$ and that
$\textrm{rank}(h_{ij})<(n-1)$. From
the Zermelo condition $\mathbb{C}_k(F)=0$ we obtain that
$h_{ij}y^{(1)j}=0$. Therefore, in view of our assumption, there exist
the functions $X^j\neq Py^{(1)j}$ that satisfy $h_{ij}X^j=0$. It
follows that the non-zero vector field $X=X^j{\partial}/{\partial y^{(2k-1)j}}$
  satisfies $i_X\omega_F$, which contradicts the assumption that
  $\textrm{rank}(\omega_F)=2k(n-1)$.  
\end{proof}
The homogeneity properties of the Poincar\'e-Cartan forms $\theta_F$
and $\omega_F$ were proven in a different context in Proposition 6.1 and Theorem 6.4
of \cite{CS04}. 

\subsection{Higher order projective metrizability} \label{subsec:hopm}

In this subsection we formulate and discuss the projective
metrizability problem for homogeneous higher order systems. We show
first that the variational problem of a Finsler function of order $k$
uniquely determines a projective class of homogeneous higher order
systems. Then, we characterize the metrizability of a homogeneous higher order
systems in terms of some homogeneous semi-basic $1$-forms. 

\begin{defn} \label{def:pm} A homogeneous semispray $S\in \mathfrak{X}(T^{2k-1}_0M)$
  is said to be  \emph{projectively metrizable} if its geodesics
  coincide with the solutions of the Euler-Lagrange equations of some
  (locally defined) Finsler function $F$, up to an orientation preserving
  reparameterization.
\end{defn}

The variational problem for a regular Lagrangian on
$T^kM$ uniquely determines a Lagrangian semispray of order $2k-1$. 
In Theorem \ref{thm:ls} we gave characterizations for a
semispray $S$ to be a Lagrangian semispray.

We will see now that in the
case of a Finsler function of order $k$, the variational problem
uniquely determine a projective class of sprays. In the next theorem,
which represents the homogeneous version on Theorem \ref{thm:ls}, we provide
characterizations of projectively metrizable homogeneous semisprays in terms of
homogeneous semi-basic $1$-forms, extending the case $k=1$ studied in
\cite[\S 4.3]{BD09}.

\begin{thm} \label{thm:pm} Consider $S$ a homogeneous semispray of order
  $2k-1$. \begin{itemize}  \item[i)]  $S$ is projectively
  metrizable if and only if  it satisfies either one of the following equivalent two
equations  
\begin{eqnarray} \mathcal{L}_S\theta_F=dF, \quad
  i_S\omega_F=0, \label{isoef} \end{eqnarray} for some (locally
defined) Finsler functions $F$ of order $k$. \item[ii)] $S$ is projectively
  metrizable if and only if there exists a (locally defined) homogeneous semi-basic $1$-form
  $\theta$ on $T^{2k-1}M$ of order $k$, such
  $\operatorname{rank}(d\theta)=2k(n-1)$ and the $1$-form $\mathcal{L}_S\theta$ is
  closed.
  \end{itemize}
\end{thm}
\begin{proof}
i) In view of the two formulae \eqref{istf} we have that the two equations \eqref{isoef} are
equivalent. 

For the direct implication, we assume that $S$ is
projectively metrizable. Then, the semispray $S$ satisfies the equation
\begin{eqnarray}
\frac{\partial F}{\partial x^i} - \frac{1}{1!} S\left(\frac{\partial
    F}{\partial y^{(1)i}}\right) + \cdots + \frac{(-1)^k}{k!} S^k\left(\frac{\partial
    F}{\partial y^{(k)i}}\right) =0, \label{efs}
\end{eqnarray} for some Finsler function $F$ on $T^kM$. Using similar
arguments as we did use in the proof of Theorem \ref{thm:ls}, it
follows that equation \eqref{efs} is equivalent to first equation \eqref{isoef}.

For the converse implication, consider $F$ a
Finsler function of order $k$. We assume that the semispray $S$ is a solution of the second
equation \eqref{isoef}.  Locally, first equation \eqref{isoef} is
equivalent to \begin{eqnarray}
(-1)^k {2k \choose k} h_{ij}G^j =  \frac{\partial F}{\partial x^i} - \frac{1}{1!} d_T\left(\frac{\partial
    F}{\partial y^{(1)i}}\right) + \cdots + \frac{(-1)^k}{k!} d^k_T\left(\frac{\partial
    F}{\partial y^{(k)i}}\right). \label{hijgj} 
\end{eqnarray} It follows that two  homogeneous semisprays $S_1$ and $S_2$ are solutions of
either one of the two equations \eqref{isoef} if and only if the
semispray coefficients $G^i_1$ and $G^i_2$ satisfy
\begin{eqnarray}
h_{ij}(G^i_1-G^i_2)=0. \label{hijg12}
\end{eqnarray}
The regularity condition for the Finsler function $F$ implies that the
only solutions of equation \eqref{hijg12} are given by
$G^i_1-G^i_2=Py^{(1)i}$, for some function $P\in
C^{\infty}(T^{2k-1}_0M)$, and hence the two  homogeneous semisprays $S_1$ and $S_2$ are
projectively equivalent.  Therefore, equations \eqref{isoef} uniquely
determine the projective class of a  homogeneous semispray $S$, and this  homogeneous semispray is
projectively metrizable. 

ii) For the first implication  we assume that the  homogeneous semispray $S$ is
projectively metrizable. Therefore, it satisfies first equation
\eqref{isoef}, for some (locally defined) Finsler function $F$ of order $k$. We
consider $\theta=\theta_F$, the Poincar\'e-Cartan $1$-form of $F$. We
have that $\theta\in \Lambda^1(T^{2k-1}_0M)$ is a homogeneous,
semi-basic $1$-form of order $k$, the $1$-form $\mathcal{L}_S \theta$ is
closed, and according to iii) of Lemma \ref{lem:fthom} we have $\textrm{rank}(d\theta)=2k(n-1)$.

For the converse implication, consider a homogeneous semi-basic $1$-form
  $\theta \in \Lambda^1(T^{2k-1}_0M)$ of order $k$, such
  $\operatorname{rank}(d\theta)=2k(n-1)$ and $\mathcal{L}_S\theta$ is
  a closed $1$-form. 

We will prove first that the condition $\mathcal{L}_S\theta$ is closed
implies $\mathcal{L}_S\theta=di_S\theta$. Since $S$ is a homogeneous
semispray  of
order $2k-1$ it follows $[J,S]S=S+P_1\mathbb{C}_{2k-1}$, for some
function $P_1 \in C^{\infty}(T^{2k-1}_0M)$. Due to the homogeneity of
the semi-basic $1$-form $\theta$ it follows that
$\mathcal{L}_{P_1\mathbb{C}_{2k-1}} \theta=0$. Using the commutation
rule \cite[1a, p.205]{GM00} it follows 
\begin{eqnarray}
\mathcal{L}_S\theta = \mathcal{L}_{[J,S]S}\theta = i_Sd_{[J,S]}\theta
+ d_{[J,S]} i_S\theta +
i_{[[J,S],S]}\theta. \label{lst1} \end{eqnarray}
Since $\theta$ is a semi-basic $1$-form of order $k$ it follows that
for $k\geq 2$ the $1$-form $\mathcal{L}_S\theta$ is semi-basic of
order $k+1\leq 2k-1$. Using  formula
\eqref{isjbt}  for $\beta=1$  and the commutation
rule \cite[1c, p.205]{GM00} it follows 
\begin{eqnarray}
i_{[[J,S],S]}\theta = i_{[J,S]}\mathcal{L}_S\theta
-\mathcal{L}_Si_{[J,S]}\theta= \mathcal{L}_S\theta -
\mathcal{L}_S\theta =0.\label{lst2} \end{eqnarray}
Formulae \eqref{dil}, \eqref{isjbt}  for $\beta=1$ and the condition
$\mathcal{L}_Sd\theta=0$ imply:
\begin{eqnarray*}
d_{[J,S]}\theta=i_{[J,S]}d\theta - di_{[J,S]}\theta = i_{[J,S]}d\theta
- d\theta = i_J\mathcal{L}_Sd\theta - \mathcal{L}_Si_Jd\theta -
d\theta = - \mathcal{L}_Si_Jd\theta -
d\theta. \end{eqnarray*}
In the above formula we compose with $i_S$, use the homogeneity
condition $i_{\mathbb{C}_1}d\theta=0$ and obtain
\begin{eqnarray}
\nonumber i_Sd_{[J,S]}\theta & =& - \mathcal{L}_Si_Si_Jd\theta - i_Sd \theta= -
\mathcal{L}_S \left( i_Ji_Sd\theta + i_{\mathbb{C}_1}d\theta\right) -
i_Sd \theta =  -
\mathcal{L}_S i_Ji_Sd\theta -
i_Sd \theta \\ 
& = & - i_J\mathcal{L}_S i_S d\theta + i_{[J,S]}i_Sd\theta  - i_Sd
\theta =  i_{[J,S]}i_Sd\theta  - i_Sd
\theta. \label{lst3} \end{eqnarray}
If we replace now formulae \eqref{lst2} and \eqref{lst3} and the fact
that $d_{[J,S]} i_S\theta = i_{[J,S]} di_S\theta$ in formula
\eqref{lst1} we obtain
\begin{eqnarray*}
\mathcal{L}_S\theta = -i_Sd\theta + i_{[J,S]}i_Sd\theta
+ i_{[J,S]} di_S\theta = -i_Sd\theta+ i_{[J,S]}\mathcal{L}_S\theta = 
-i_Sd\theta+ \mathcal{L}_S\theta = di_S\theta.\label{lst4} \end{eqnarray*}
Consider the function $F=i_S\theta$. Above formula shows that the
function $F$ satisfies formula \eqref{lstdl}, for $L=F$, which means
$\mathcal{L}_S\theta =dF$. Since $\mathcal{L}_S\theta\in
\Lambda^1(T^{2k-1}_0M)$ is semi-basic of order $(k+1)$, 
it follows that the function $F$ is
constant along the fibres of the projection $\pi^{2k-1}_k: T^{2k-1}_0M
\to T^k_0M$ and hence we can assume that $F \in
C^{\infty}(T^k_0M)$. Using part v) of Lemma \ref{lem:cas} we obtain that $\theta=\theta_F$. 

We have to prove now that the function $F$ is a
Finsler function. From first formulae \eqref{ictl} it follows that
$i_{\mathbb{C}_{\alpha}}\theta_F=0$ if and only if
$\mathbb{C}_{\alpha+1}(F)=0$ for all $\alpha \in
\{1,....,k-1\}$. Since $\theta$ is homogeneous, we obtain
$\mathbb{C}_2(F)=\cdots =\mathbb{C}_k(F)=0$. These arguments, the
definition of function $F$ and
formula \eqref{istl} imply $F=i_S\theta_F=\mathbb{C}_1(F)$. It follows
  that Zermelo conditions \eqref{zermelo} are satisfied. Finally, we
  have that $\operatorname{rank}(d\theta)=2k(n-1)$, and using part iii) of Lemma \ref{lem:fthom} implies that $F$ is a
  Finsler function of order $k$. Now the condition
  $\mathcal{L}_S\theta_F=dF$ says that $S$ is projectively metrizable.
\end{proof} 

For $k=1$, second equation \eqref{isoef} reduces to Rapcs\'ak equation
\cite[Rap 1]{szilasi07}.

We note that for the converse implication of the first part of Theorem
\ref{thm:pm} we do not need the requirement that the semispray $S$ is
homogeneous. The argument is as follows, and it is due to Crampin and
Saunders \cite[Thm. 4.4]{CS11}. For a semispray $S$, solution of
second equation \eqref{isoef}, using formulae \eqref{icaof} it follows that
$\mathcal{D}=\textrm{span} \{S, \mathbb{C}_1, ....,
\mathbb{C}_{2k-1}\}=\operatorname{Ker} \omega_F$. Since the
Poincar\'e-Cartan $2$-form $\omega_F=-d\theta_F$ is closed it follows
that its  characteristic distribution $\mathcal{D}$ is involutive and
hence $S$ is a  homogeneous semispray.

\section{Examples} \label{sec:examples}

For a Finsler function $F$, of order $k\geq 1$, its variational problem
uniquely determines a projective class of homogeneous semisprays,
solutions of either one of the two equivalent equations \eqref{isoef}. 

For $k=1$, in this projective class of homogeneous semisprays, we can
single out one spray, which is called the geodesic spray. The geodesic
spray is the only semispray determined by the variational problem of
the regular Lagrangian $L=F^2$. Moreover, the geodesic spray is the
only spray, in the projective class determined by the Finsler function
$F$, whose geodesics are parameterized by arc-length. 

When $k>1$, we do not know yet if it is possible, and eventually how, to associate to a
Finsler function $F$ of order $k$, a regular Lagrangian of
order $k$. Therefore, the only option to fix a homogeneous semispray, which was suggested to me by
David Saunders, in the projective class determined by the variational problem of $F$,
is to use the arc-length induced by $F$.

Next we use some examples to discuss the above considerations as well
as the results obtained in the previous sections. 

For a Riemannian metric $g_{ij}(x)$ on a manifold $M$, consider the
functions $L_1, F_1: TM \to \mathbb{R}$, given by
\begin{eqnarray}
L_1(x, y^{(1)})=\frac{1}{2}g_{ij}(x)y^{(1)i}y^{(1)j}=
\frac{1}{2}\|y^{(1)}\|^2_g, \ 
F_1(x,y^{(1)})= \sqrt{g_{ij}(x)y^{(1)i}y^{(1)j}} = \|y^{(1)}\|_g. \label{L1F1} \end{eqnarray}

$L_1$ is a regular Lagrangian of order one, its
Hessian matrix, given by formula \eqref{gij}, is just the Riemannian
metric $g_{ij}(x)$. The variational problem for $L_1$ uniquely
determines a spray $S_1\in \mathfrak{X}(TM)$, which is the geodesic spray for the
Riemannian metric $g_{ij}(x)$. The geodesic spray $S_1$ is uniquely
determined by either one of the two equations \eqref{isoel}, for $k=1$,
and it is given by
\begin{eqnarray}
S_1=y^{(1)i}\frac{\partial}{\partial x^i} -
\gamma^i_{jk}(x)y^{(1)j}y^{(1)k} \frac{\partial}{\partial
  y^{(1)i}}, \label{gs1} \end{eqnarray}
where $\gamma^i_{jk}(x)$ are the Christoffel symbols of the Riemannian
metric $g_{ij}(x)$.

$F_1$ is a Finsler function of order $1$, its angular metric tensor, given by
formula \eqref{hij}, is related to the Riemannian metric as follows 
\begin{eqnarray}
h^{(1)}_{ij}(x,y^{(1)})=F_1\frac{\partial^2 F_1}{\partial
  y^{(1)i}\partial y^{(1)j}}  = g_{ij}(x)-\frac{\partial F_1}{\partial
  y^{(1)i}}(x,y^{(1)}) \frac{\partial F_1}{\partial y^{(1)j}}(x,y^{(1)}). \label{hgij} \end{eqnarray}  
We have that $\textrm{rank}(h^{(1)}_{ij})=n-1$ and the variational problem
for $F_1$ uniquely determines the projective class of the geodesic
spray $S_1$. Within this projective class, $S_1$ is the only spray
whose geodesics are parameterized by the arc-length of the given
riemannian metric $g$.

On the second order tangent bundle $T^2M$, we consider the locally
defined functions
\begin{eqnarray*}
z^{(2)i}(x, y^{(1)}, y^{(2)})= y^{(2)i}+ \frac{1}{2}\gamma^i_{jk}(x)y^{(1)j}y^{(1)k}.
\label{z2}
\end{eqnarray*}
It follows that $z^{(2)i}$ behave as the components of a vector
field on $M$. These components were interpreted as the covariant
form of acceleration in \cite[(6.5)]{BM09}, as half of the
components of the tension field in \cite{CMOP06}. It follows that the
function $L_2: T^2M \to \mathbb{R}$, given by
\begin{eqnarray*}
L_2(x,y^{(1)}, y^{(2)})=\frac{1}{2}g_{ij}(x) z^{(2)i}(x, y^{(1)},
y^{(2)}) z^{(2)j}(x, y^{(1)}, y^{(2)})=\frac{1}{2}\|z^{(2)}\|^2_g, \label{l2}
\end{eqnarray*} is a second order regular Lagrangian. The Hessian matrix of $L_2$, given by formula \eqref{gij}, is the Riemannian
metric $g_{ij}$. The variational
problem for $L_2$ uniquely determines a semispray of order $3$, $S_3\in
\mathfrak{X}(T^3M)$, whose geodesics are biharmonic curves
\cite{CMOP06}. We call $S_3$ the biharmonic semispray and it is
uniquely determined by either one of the two equivalent equations
\eqref{isoel}. The local coefficients of the biharmonic semispray can
be determined as in \cite[(4.6)]{BCD11}, while the biharmonic equations can
be written as in \cite[(4.8)]{BCD11}. 

For the second order Lagrangian $L_2$, the following homogeneity
properties are true: 
\begin{eqnarray} \mathbb{C}_1(L_2)=4L_2, \quad
  \mathbb{C}_2(L_2)=g_{ij}(x)z^{(2)i}y^{(1)j} =\frac{1}{2}S_1(L_1). \label{c1c2l2} \end{eqnarray} 
Using Proposition \ref{prop:homls} it follows that the biharmonic semispray $S_3$
satisfies the homogeneity condition  $[\mathbb{C}_1,
S_3]=S_3$. However, the biharmonic semispray $S_3$ is not a homogenous
semispray since it does not satisfy the equation \eqref{rhspray} for $\alpha=3$.

Consider the function $F_2:T^2_0M \to \mathbb{R}$, 
\begin{eqnarray} 
F_2=\frac{2 F^2_1 L_2 -  (\mathbb{C}_2(L_2))^2}{F_1^5}=\frac{\|z^{(2)}\|_g^2 \|y^{(1)}\|^2_g
-\left(g_{ij} y^{(1)i} z^{(2)j}\right)^2}{\|y^{(1)}\|^5_g}. \label{f2}\end{eqnarray}
The numerator of the right hand side of the above formula is $\|z^{(2)}\|_g^2 \|y^{(1)}\|^2_g
-\left(g_{ij} y^{(1)i} z^{(2)j}\right)^2 \geq 0$, and hence $F_2\geq 0$ on
$T^2_0M$. Using the homogeneity properties \eqref{c1c2l2} of the
Lagrangian $L_2$, we obtain
that $F_2$ satisfies the Zermelo conditions \eqref{zermelo}, for
$k=2$. Moreover, the tensor \eqref{hij} that corresponds to $F_2$ is
given by 
\begin{eqnarray}
h^{(2)}_{ij}=F^3_2\frac{\partial^2 F_2}{\partial y^{(2)i}\partial
  y^{(2)j}} = 2\left(\frac{F_2}{F_1}\right)^3h^{(1)}_{ij}. \label{h2h1}
\end{eqnarray}
It follows that $\textrm{rank}(h^{(2)}_{ij})=
\textrm{rank}(h^{(1)}_{ij}) = n-1$ and therefore $F_2$ is a Finsler
function of order $2$. Using formulae \eqref{hgij}, \eqref{f2}, and \eqref{h2h1}, it follows that one
can recover the Finsler function of order $2$, $F_2$, from
either one of the angular metrics $h^{(1)}_{ij}$ or $h^{(2)}_{ij}$ as follows
\begin{eqnarray}
\nonumber F_2(x, y^{(1)}, y^{(2)}) &=&
\frac{1}{F^3_1}h^{(1)}_{ij}(x,y^{(1)})z^{(2)i}z^{(2)j} \\
& =& \sqrt[4]{\frac{1}{2} h^{(2)}_{ij}(x,y^{(1)}, y^{(2)})z^{(2)i}z^{(2)j}}. \label{f2f1}\end{eqnarray}
 In the Euclidean context, $F_2$ reduces to the
parametric Lagrangian considered by Crampin and Saunders in
\cite{CS11}. The function $F_2/F_1 \in C^{\infty}(T^2_0M)$, which 
connects the angular metrics of the two Finsler functions, is related
to the first curvature $\kappa$ of a curve. Indeed we have
$F_2/F_1=\kappa^2$. See also formula (39) in \cite{matsyuk11} for $A=0$.

The variational problem for $F_2$ uniquely determines a system of
fourth order differential equations, which is invariant under
orientation preserving reparameterizations. By fixing the parameter to
be the arc-length, the system reduces to the dynamical equation of motion (38) 
studied by Matsyuk \cite{matsyuk11}. In the Euclidean context, a
homogeneous semispray, in the projective class determined by the
variational problem of $F_2$ was obtained in \cite{CS11}.

\begin{acknowledgement*}
The author expresses his thanks to the reviewers for their comments and
suggestions, especially to those regarding the regularity condition for higher order Lagrange
and Finsler functions.

This work has been supported by the Romanian National
Authority for Scientific Research, CNCS UEFISCDI, project number
PN-II-RU-TE-2011-3-0017. 
\end{acknowledgement*}

\end{document}